\documentclass{amsart}
\usepackage{amssymb,amsrefs}
\usepackage[colorlinks]{hyperref}
\usepackage{booktabs}
\usepackage[all]{xy}
\usepackage{graphicx}

\DeclareMathAlphabet{\mathpzc}{OT1}{pzc}{m}{it}

\theoremstyle{plain}
\newtheorem{theorem}{Theorem}[section]
\newtheorem{proposition}[theorem]{Proposition}
\newtheorem{lemma}[theorem]{Lemma}
\newtheorem{corollary}[theorem]{Corollary}

\newtheorem{heuristic}[theorem]{Heuristic}
\newtheorem{expectation}[theorem]{Heuristic Expectation}

\theoremstyle{definition}
\newtheorem{remark}[theorem]{Remark}
\newtheorem{definition}[theorem]{Definition}

\newtheorem{algorithm}[theorem]{Algorithm}
\newtheorem{notation}[theorem]{Notation}

\DeclareMathOperator{\Cl}{Cl}
\DeclareMathOperator{\degree}{deg}
\DeclareMathOperator{\defect}{defect}
\DeclareMathOperator{\End}{End}
\DeclareMathOperator{\Hom}{Hom}
\DeclareMathOperator{\Jac}{Jac}
\DeclareMathOperator{\Otilde}{\tilde{O}}
\DeclareMathOperator{\pol}{\scalebox{1.2}{$\mathpzc{Pols}$}}
\DeclareMathOperator{\cur}{\scalebox{1.2}{$\mathpzc{Curves}$}}

\newcommand{\FF}{{\mathbf{F}}}
\newcommand{\HH}{{\mathbf{H}}}
\newcommand{\PP}{{\mathbf{P}}}
\newcommand{\QQ}{{\mathbf{Q}}}
\newcommand{\ZZ}{{\mathbf{Z}}}

\newcommand{\frakp}{{\mathfrak p}}

\newcommand{\calA}{{\mathcal{A}}}
\newcommand{\calAbar}{\bar{\mathcal{A}}}
\newcommand{\calC}{{\mathcal{C}}}
\newcommand{\calD}{{\mathcal{D}}}
\newcommand{\calE}{{\mathcal{E}}}
\newcommand{\calL}{{\mathcal{L}}}
\newcommand{\calO}{{\mathcal{O}}}
\newcommand{\calS}{{\mathcal{S}}}
\newcommand{\calObar}{\bar{\mathcal{O}}}

\newcommand{\eps}{\varepsilon}
\newcommand{\Fq}{\FF_q}
\newcommand{\Fqbar}{\bar{\FF}_q}
\newcommand{\Kbar}{\bar{K}}
\newcommand{\upleftparen}{\textup{\char40}}
\newcommand{\uprightparen}{\textup{\char41}}

\renewcommand{\hat}{\widehat}
\renewcommand{\tilde}{\widetilde}
\renewcommand{\mod}{\bmod}

\providecommand{\abs}[1]{\lvert#1\rvert}

\newcommand\lowtilde{\lower0.7ex\hbox{\textasciitilde}}

\newcommand{\mybar}[1]{
  \mathchoice
  {#1\llap{$\overline{\phantom{\displaystyle\rm#1}}$}}
  {#1\llap{$\overline{\phantom{\textstyle\rm#1}}$}}
  {#1\llap{$\overline{\phantom{\scriptstyle\rm#1}}$}}
  {#1\llap{$\overline{\phantom{\scriptscriptstyle\rm#1}}$}}
}  
\renewcommand{\bar}{\mybar}

\newlength{\algindent}
\newlength{\alglabel}
\newcounter{stepcount}

\newenvironment{algtop}
{\quad\begin{list}{\arabic{stepcount}.}{\leftmargin=\algindent\labelwidth=\algindent\itemsep=\smallskipamount\usecounter{stepcount}}}
{\end{list}}

\newcommand{\algin}{\item[\emph{Input}:]}
\newcommand{\algout}{\item[\emph{Output}:]}

\newenvironment{alglist}
{\quad\begin{list}{\arabic{stepcount}.}{\leftmargin=1.5em\labelwidth=1em\labelsep0.5em\itemsep=\smallskipamount\usecounter{stepcount}}}
{\end{list}}

\newcounter{substepcount}
\newenvironment{algsublist}
{\quad\begin{list}{(\/\rlap{\alph{substepcount}}\phantom{d}\/)}{\usecounter{substepcount}}}
{\end{list}}

\begin{document}

\title[Curves of genus $4$ with many points]
{Quickly constructing \\ curves of genus $4$ with many points}

\author{Everett W. Howe} 
\address{Center for Communications Research,
         4320 Westerra Court,
         San Diego, CA 92121, USA}
\email{however@alumni.caltech.edu}
\urladdr{http://www.alumni.caltech.edu/\lowtilde{}however/}

\date{14 June 2015}


\begin{abstract}
The \emph{defect} of a curve over a finite field is the difference between the 
number of rational points on the curve and the Weil--Serre bound for the curve.
We present a construction for producing genus-$4$ double covers of genus-$2$ 
curves over finite fields such that the defect of the double cover is not much 
more than the defect of the genus-$2$ curve. We give an algorithm that uses 
this construction to produce genus-$4$ curves with small defect. Heuristically,
for all sufficiently large primes and for almost all prime powers $q$, the 
algorithm is expected to produce a genus-$4$ curve over $\Fq$ with defect at
most $4$ in time $\Otilde(q^{3/4})$.

As part of the analysis of the algorithm, we present a reinterpretation of 
results of Hayashida on the number of genus-$2$ curves whose Jacobians are 
isomorphic to the square of a given elliptic curve with complex multiplication
by a maximal order.  We show that a category of principal polarizations on the 
square of such an elliptic curve is equivalent to a category of right ideals in 
a certain quaternion order.
\end{abstract}

\maketitle

\section{Introduction}
\label{S:intro}

For every prime power $q$ and non-negative integer $g$, we let $N_q(g)$ denote
the maximum number of rational points on a smooth, projective, absolutely
irreducible curve of genus $g$ over the finite field~$\Fq$. There are many 
interesting questions one might ask about the asymptotic rate of growth of
$N_q(g)$, but in this paper we will focus on the problem of determining, for
specific $q$ and $g$, reasonably tight upper and lower bounds on~$N_q(g)$. 
In particular, we will consider the case where $g=4$.

The Riemann Hypothese for curves over finite fields, proven in the 1940's by
Weil~\cites{Weil1940,Weil1941,Weil1945,Weil1946}, shows that for every 
genus-$g$ curve $C$ over $\Fq$ we have 
\[ q+1-2g\sqrt{q} \le \#C(\Fq) \le q+1+2g\sqrt{q}.\]
This gives the upper bound $N_q(g) \le q+1+2g\sqrt{q}.$ Serre~\cite{Serre1983a}
improved this bound for nonsquare $q$ by showing that in fact
\[ N_q(g) \le q + 1 + g \lfloor 2\sqrt{q}\rfloor.\]
When $q \ge 2g + \sqrt{2g} + 1$, this \emph{Weil--Serre upper bound} is usually 
the best upper bound we know for $N_q(g)$; for certain proper prime powers 
$q\ge 2g+ \sqrt{2g} + 1$ there are improvements that can be made (see, for 
example, Theorem~4, Proposition~5, and Corollary~6 of~\cite{HoweLauter2003}),
but for primes in this range no such improvements are known.

Fifteen years ago, van der Geer and van der Vlugt~\cite{GeerVlugt2000}
published a table of the best upper and lower bounds on $N_q(g)$ known at the
time, for $g\le 50$ and for $q$ ranging over small powers of $2$ and~$3$. They 
regularly updated their tables and posted the updates on van der Geer's
website. Ten years after the first publication of these tables, the website
\url{manypoints.org} was created by van der Geer, Lauter, Ritzenthaler, and the
author (with technical assistance from Geerit Oomens). The \texttt{manypoints} 
tables include results for many more prime powers $q$ than were 
in~\cite{GeerVlugt2000}: namely, the primes less than~$100$, the prime powers 
$p^i$ for $p<20$ and $i\le 5$, and the powers of $2$ up to~$2^7$. 

The upper bounds presented in the \texttt{manypoints} tables come from a wide 
variety of sources and techniques, as is explained in the introduction 
to~\cite{HoweLauter2012}.  But lower bounds for $N_q(g)$ are generally obtained
by producing examples of curves with many points, and for most $q$ and $g$
no-one has done thorough searches for such curves.

For $g=1$, the value of $N_q(g)$ is determined by a result of
Deuring~\cite{Deuring1941} (see~\cite{Waterhouse1969}*{Theorem~4.1, p.~536}).
For $g=2$, the value of $N_q(g)$ is given by a result of
Serre~\cites{Serre1983a,Serre1983b,Serre1984} 
(see also~\cite{HoweNartEtAl2009}). As is explained 
in~\cite{LachaudRitzenthalerEtAl2010}, there is no easy formula known 
for~$N_q(3)$, but for all $q$ in the \texttt{manypoints} tables the value has
been computed; the introduction to~\cite{Mestre2010} gives a good summary of
the techniques that have been used to find genus-$3$ curves attaining the 
maximum number of points.

This leads us to the case $g=4$.  In an earlier paper~\cite{Howe2012}, we 
obtained new upper and lower bounds on $N_q(4)$ for those prime powers $q<100$
for which the exact value had not been known. The strategy we used for small 
$q$ sometimes required us to search through families of curves to determine 
whether any members of the family had many points.  Such strategies do not 
scale well with the size of the base field.

In this paper, we take a different tack.  Instead of performing exhaustive
searches to try to find curves with point-counts as close as possible to the 
Weil--Serre bound, we will exhibit an \emph{efficient} algorithm that we expect
will produce curves whose point-counts are \emph{reasonably close} to the
Weil--Serre bound. In particular, heuristic arguments suggest that for
all sufficiently large primes and for most prime powers~$q$, our 
Algorithm~\ref{A:genus4} will produce a genus-$4$ curve over $\Fq$ whose number
of points is within $4$ of the Weil bound in time $\Otilde(q^{3/4})$.

In Section~\ref{S:family} we give Algorithm~\ref{A:construction}, which 
attempts to produce genus-$4$ curves over $\Fq$ whose Jacobians split (up to a
small isogeny) as a product of a given genus-$2$ Jacobian and two elliptic 
curves from a given list. The algorithm takes as input a hyperelliptic curve 
that can be written $y^2 = f_1 f_2$ for two cubic polynomials $f_1$ and $f_2$,
and two lists $\calL_1$ and $\calL_2$ of elliptic curves that are `compatible'
(in a sense defined in Section~\ref{S:family}) with the polynomials $f_1$ 
and~$f_2$.  Heuristically, we expect the algorithm to succeed with probability
proportional to $\#\calL_1\#\calL_2/q$, and to run in time linear in $\#\calL_1$
and polynomial in~$\log q$.

The \emph{defect} of a genus-$g$ curve $C$ over $\Fq$ is the difference
between the Weil--Serre bound and the number of points on $C$:
\[
\defect C = (q + 1 + g\lfloor 2\sqrt{q}\rfloor) - \#C(\Fq).
\]
In Section~\ref{S:defect}, we explain how Algorithm~\ref{A:construction} can be
used to produce genus-$4$ covers $D$ of genus-$2$ curves $C$ such that the 
defect of $D$ is not much larger than the defect of $C$.

In Section~\ref{S:Hayashida}, which may be of independent interest, we
reinterpret some results of Hayashida having to do with principal polarizations
on the square of an elliptic curve with complex multiplication.  We use these
results in Section~\ref{S:genus2} to show that for many primes $q$ there are
many genus-$2$ curves of small defect over $\FF_q$, and we give an algorithm 
for finding them.

In Section~\ref{S:genus4} we put all of the pieces together and present
Algorithm~\ref{A:genus4}, which, as we noted earlier, can be used to quickly
find curves of genus $4$ with small defect.  In Section~\ref{S:results} we
show how Algorithm~\ref{A:genus4} fares in actual practice.  

\subsection*{Acknowledgments}
The author is grateful to John Voight for sharing a draft of his forthcoming
book~\cite{Voight} and for providing the argument involving quaternion algebras
that appears in the proof of Theorem~\ref{T:whenitworks}.

\section{A family of genus-\texorpdfstring{$4$}{4} curves covering a 
         genus-\texorpdfstring{$2$}{2} curve}
\label{S:family}

Let $k$ be a finite field of odd characteristic and let $C$ be a genus-$2$
curve over~$k$, given by a model $y^2 = f$ where $f\in k[x]$ is a separable 
polynomial of degree~$6$.  Suppose $f$ can be written $f = f_1 f_2$, where 
$f_1$ and $f_2$ are cubic polynomials. We will associate to this factorization 
of $f$ a $1$-parameter family of genus-$4$ curves over $k$ that are double
covers of $C$ and whose Jacobians are isogenous to the product of the Jacobian 
of $C$ with two (variable) elliptic curves.

For every $a\in \PP^1(k)$ with $a\ne\infty$ and $f(a)\ne 0$, let $D_a$ be the
curve defined by the pair of equations
\begin{align*}
w^2 &= (x-a) f_1\\
z^2 &= (x-a) f_2.
\end{align*}
For $i=1$ and $i=2$ let $h_i$ be the polynomial $ x^3 f_i(1/x+a)$ and let 
$E_{a,i}$ be the elliptic curve $y^2 = h_i$, so that $E_{a,i}$ is isomorphic
to the genus-$1$ curve $y^2 = (x-a)f_i$.  To handle the case $a=\infty$, 
we let $D_\infty$ be the curve defined by $w^2 = f_1, z^2 = f_2$, and for each 
$i$ we let $E_{\infty,i}$ be the elliptic curve $y^2 = f_i$.  If we have call
to make the dependence of $D_a$ on $f_1$ and $f_2$ explicit, we will write 
$D_a(f_1,f_2)$.

\begin{theorem}
\label{T:basic4}
For each $a\in\PP^1(k)$ such that $f(a)\ne 0$ the curve $D_a$ has genus~$4$,
and there are isogenies
\begin{align*}
\varphi&\colon \Jac D_a \to (\Jac C)\times E_{a,1}\times E_{a,2}\\
\psi   &\colon (\Jac C)\times E_{a,1}\times E_{a,2} \to \Jac D_a
\end{align*}
such that $\varphi\circ\psi$ and $\psi\circ\varphi$ are multiplication-by-$2$.
\end{theorem}

\begin{proof}
Let $K = k(x)$, so that the function field of $D_a$ is $K(w,z)$.  Since
\[
\left(\frac{wz}{x-a}\right)^2 = f_1 f_2 = f,
\]
we see that $K(wz/(x-a))$ is isomorphic to the function field of $C$. The 
diagram of Galois field extensions on the left then leads to the diagram of 
curves on the right:
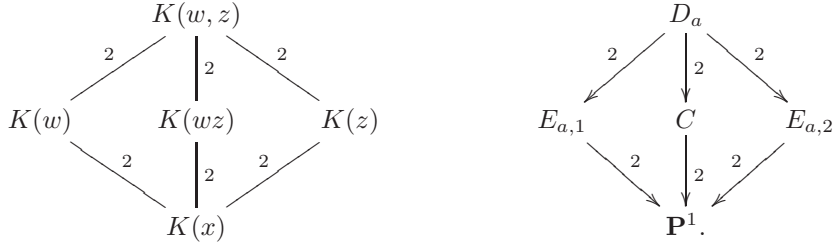
\[
\xymatrix{
                  & K(w,z)\ar@{-}[dl]_2\ar@{-}[dr]^2 &                   & &                  & D_a\ar[dl]_2\ar[dr]^2   &                 \\
K(w)\ar@{-}[dr]^2 & K(wz) \ar@{-}[d]^2 \ar@{-}[u]_2  & K(z)\ar@{-}[dl]_2 & & E_{a,1}\ar[dr]^2 & C \ar@{<-}[u]_2\ar[d]^2 & E_{a,2}\ar[dl]_2\\
                  & K(x)                             &                   & &                  & \PP^1.                  &                 \\
}
\]
The Galois group of $D_a$ over $\PP^1$ is of course the Klein group~$V_4$.  The 
images of the pullback maps from $\Jac C$, $E_{a,1}$, and $E_{a,2}$ to 
$\Jac D_a$ are subvarieties of $\Jac D_a$ that have $0$-dimensional pairwise 
intersections, because a different subgroup of $V_4$ acts trivially on each of 
the three subvarieties. Therefore, the three pullback maps piece together to 
give an isogeny $\psi\colon (\Jac C)\times E_{a,1}\times E_{a,2}\to \Jac D_a$, 
and the pushforwards give an isogeny $\varphi$ in the other direction. On each 
factor, the composition $\varphi\circ\psi$ is multiplication by~$2$, so
$\varphi\circ\psi = 2.$  The composition in the other order must then also be
multiplication by~$2$.
\end{proof}

\begin{definition}
\label{D:compatible}
Let $f\in k[x]$ be a separable cubic.  We say that an elliptic curve $E$ over
$k$ is \emph{compatible} with $f$ if $E$ is isomorphic to the curve $y^2 = cf$
for some nonzero $c\in k$ or if $E$ is isomorphic to the curve 
$y^2 = c x^3 f(1/x + a)$ for some $a,c\in k$ with $c\ne 0$.
\end{definition}

\begin{lemma}
An elliptic curve $E$ over $k$ is compatible with a separable cubic $f\in k[x]$
if and only if $E$ and the elliptic curve $y^2 = f$ have the same number of
$k$-rational $2$-torsion points.
\end{lemma}

\begin{proof}
Write $E$ as $y^2 = g$ for some separable cubic~$g$.   Then $E$ will have the
same number of rational $2$-torsion points as $y^2 = f$ if and only if the 
degrees of the irreducible factors of $f$ and of $g$ are equal. Likewise, these
degrees will be equal if and only if there is a linear fractional transformation
of $\PP^1_k$ that takes the roots of $f$ to the roots of~$g$.

Suppose there is such a linear fractional transformation.  If it is of
the form $x\mapsto b x + d$, then $E$ is isomorphic to $y^2 = cf$ for some $c$.
If it is of the form $x \mapsto (bx + d)/(x - a)$, then $E$ is isomorphic
to $y^2 = c x^3 f(1/x + a)$ for some $c$.  

Conversely, if $E$ is isomorphic to $y^2 = cf$ or $y^2 = c x^3 f(1/x + a)$,
then clearly the factorization patterns of $f$ and $g$ are equal, so $E$
and the curve $y^2 = f$ have the same number of rational $2$-torsion points.
\end{proof}

\begin{notation}
Let $f\in k[x]$ be a separable cubic.  We let $n(f)$ denote the number of 
automorphisms of $\PP^1_k$ that permute the roots of~$f$.
\end{notation}

It is easy to see that $n(f)$ is either $3$, $2$, or $6$, depending on whether 
$f$ has $0$, $1$, or $3$ rational roots.

\begin{theorem}
\label{T:eachE}
Let $C$ and $f = f_1 f_2$ be as above.  Let $S$ be the set of roots of $f$ 
in~$k$, and let $T$ be the set $\{j(E)\}$, where $E$ ranges over all elliptic 
curves over $k$ compatible with~$f_1$.  The map from $\PP^1(k)\setminus S$ to
$k$ that sends $a$ to $j(E_{a,1})$ has image contained in $T$.  With at most 
$5$ exceptions, every element of $T$ has exactly $n(f_1)$ preimages\textup{;} 
the exceptions have fewer than $n(f_1)$ preimages.
\end{theorem}

\begin{proof}
The image of the map lies in $T$ by the very definition of compatibility.

Let $E$ be an elliptic curve compatible with $f_1$, and suppose $g\in k[x]$ is
a cubic polynomial such that $E$ is isomorphic to the curve $y^2 = g$. An 
element $a\in \PP^1(k)\setminus S$ is a preimage for $j(E)$ if and only if
there is an automorphism of $\PP^1$ that takes the roots of $g$ to the roots of
$f_1$ and that sends $\infty$ to~$a$. The number of automorphisms of $\PP^1$ 
taking the roots of $g$ to the roots of $f_1$ is equal to $n(f_1)$. These
automorphisms will take $\infty$ to distinct elements of $\PP^1$ unless $E$ has
more than $2$ automorphisms; that is, unless $j=0$ or $j=1728$.  These distinct 
elements will all lie in $\PP^1(k)\setminus S$, unless $j$ is the $j$-invariant
of one of the (at most three) curves $y^2 = (x-a)f_1$, for $a$ a root of $f_2$.
Thus, all but at most five values of $j$ in $T$ will have exactly $n(f_1)$ 
preimages.
\end{proof}

Theorem~\ref{T:eachE} says that the $j$-invariants of the curves $E_{a,1}$ that
we get from a given splitting $f = f_1 f_2$ are essentially distributed
uniformly at random from among the $j$-invariants of the elliptic curves over 
$k$ compatible with~$f_1$. Also, if a given elliptic curve is obtained as 
$E_{a,1}$ for a given splitting $f = f_1 f_2$, then its quadratic twist is 
obtained as $E_{a,1}$ for the \emph{same value of $a$} from the splitting 
$f = (c f_1) ( f_2/c)$, where $c$ is a nonsquare in~$k$. These observations 
lead us to the following heuristic.

\begin{heuristic}
\label{H:Es}
For a given curve $C$ and polynomials $f_1, f_2$ as above, we will model the 
pairs $(E_{a,1}, E_{a,2})$ as being chosen uniformly at random from among all
pairs $(E_1,E_2)$ of elliptic curves over $k$ compatible with $f_1$ and $f_2$, 
respectively.
\end{heuristic}

Now let us use the construction implicit in Theorem~\ref{T:basic4} to create an
algorithm for producing genus-$4$ double covers $D_a$ of a genus-$2$ curve $C$,
as above, where the curves $E_{a,1}$ and $E_{a,2}$ lie in a prescribed set of 
elliptic curves.

\goodbreak
\begin{algorithm}
\label{A:construction}
\begin{algtop}
\algin  An odd prime power $q$, coprime cubic polynomials $f_1$ and $f_2$
        in $\Fq[x]$, and two lists $\calL_1$ and $\calL_2$ of elliptic curves 
        over $\Fq$ compatible with $f_1$ and $f_2$, respectively.
\algout Either the word ``failure'', or a value of $a\in \PP^1(\Fq)$ such that
        $f_1(a)f_2(a)\ne 0$ and such that the elliptic curves $E_{a,1}$ and 
        $E_{a,2}$ lie in $\calL_1$ and $\calL_2$, respectively.
\end{algtop}
\begin{alglist}
\item Compute the degree-$6$ rational function $j_1\in \Fq(t)$ such that
      the $j$-invariant of the genus-$1$ curve $y^2 = (x-a)f_1$ is $j_1(a)$.
\item For every $E_1 \in \calL_1$ do:
\begin{algsublist}
\item Compute the (at most $6$) values $a\in \PP^1(\Fq)$ with $j_1(a) = j(E_1)$
      and with $f_2(a)\ne 0$.
\item For each of these values, check whether $E_{a,1}$ lies in $\calL_1$
      and $E_{a,2}$ lies in $\calL_2$.  If so, output $a$ and stop.
\end{algsublist}
\item Output ``failure''.
\end{alglist}
\end{algorithm}

Note that if there does exist an $a$ such that $E_{a,1}$ and $E_{a,2}$ lie in 
$\calL_1$ and $\calL_2$, Algorithm~\ref{A:construction} will find it.  Also, it
is clear that there are positive constants $c_1$ and $c_2$ such that 
Algorithm~\ref{A:construction} runs in probabilistic time at most 
$c_1 \#\calL_1 (\log q)^{c_2}$.

\begin{expectation}
\label{HE:success}
Let $q$ be an odd prime power.  Let $\calL_1$ and $\calL_2$ be two nonempty
lists of elliptic curves over $\Fq$ such that all the curves in each list have
the same number of $2$-torsion points, and suppose 
$\#\calL_1 \, \#\calL_2 \ll q^{3/4}$. The probability that
Algorithm~\textup{\ref{A:construction}} will succeed on a randomly chosen pair
$(f_1,f_2)$ of cubic polynomials in $\Fq[x]$ compatible with the curves in 
$\calL_1$ and $\calL_2$ is approximately
\[
n(f_1) n(f_2) \ \frac{\#\calL_1\, \#\calL_2}{4 q}.
\]
\end{expectation}

\begin{proof}[Justification]
Using Heuristic~\ref{H:Es}, we view the pairs $(E_{1,a}, E_{2,a})$ as being 
chosen uniformly at random from among the ordered pairs of elliptic curves 
compatible with $f_1$ and~$f_2$. There are approximately $4q^2/(n(f_1)n(f_2))$
such pairs of compatible curves.  The probability that none of the pairs 
$(E_{1,a}, E_{2,a})$ will lie in the set of $\#\calL_1\,\#\calL_2$ pairs we are
hoping to find is then given by
\[
\left(1 - n(f_1) n(f_2) \frac{\#\calL_1\,\#\calL_2}{4q^2}\right)^k,
\]
where $k\approx q$ is the number of elements of $\PP^1(k)$ that are not roots
of $f_1 f_2$.  Since $\#\calL_1 \,\# \calL_2 \ll q^{3/4}$, this probability of
failure is approximately 
\[ 1 - n(f_1) n(f_2) \frac{\#\calL_1\,\#\calL_2}{4q},\]
and the probability of success is as stated in the Expectation.
\end{proof}

\section{Change in defect}
\label{S:defect}

Recall that the \emph{defect} of a genus-$g$ curve $C$ over a finite field $k$
is the difference between $\#C(k)$ and the Weil--Serre upper bound for genus-$g$
curves over~$k$.  In this section, we consider using 
Algorithm~\ref{A:construction} to produce genus-$4$ curves whose defect is not
much more than that of the genus-$2$ curves that they cover.

\begin{expectation}
\label{HE:norootcubics}
Let $C$ be a genus-$2$ curve over a finite prime field~$\Fq$ that can be 
written in the form $y^2 = f_1 f_2$, where $f_1$ and $f_2$ are irreducible 
cubic polynomials in~$\Fq[x]$. Let $m = \lfloor 2\sqrt{q}\rfloor$. If $m$ is
even set $k = 1$\textup{;} if $m$ is odd set $k = 2$. Up to powers of 
$\log \log q$, the probability that there is a double cover $D$ of~$C$, of the
type described in Section~\textup{\ref{S:family}}, such that the defect of $D$ 
satisfies
\[\defect D \le \defect C + 2k\]
is approximately $q^{-1/2}$.
\end{expectation}

\begin{proof}[Justification]
We start with some comments about the number of elliptic curves with a given 
small defect.  For elliptic curves, the general Weil--Serre bound specializes
into the Hasse bound: The maximal number of points on an elliptic curve over 
$\Fq$ is $q + 1 + m$, where $m = \lfloor 2\sqrt{q} \rfloor$. If $m$ is coprime 
to $q$, then there do exist elliptic curves over $\Fq$ with this number of
points (see~\cite{Waterhouse1969}*{Theorem~4.1, p.~426}).  More generally, 
if $t$ is any integer with $\abs{t}\le \abs{m}$ and $(t,q)=1$, then the number of 
elliptic curves over $\Fq$ with $q + 1 - t$ points is equal to the Kronecker
class number $H(t^2 - 4q)$; see \cite{Schoof1987}*{Theorem~4.6, pp.~194--195}.

Let us consider the number of elliptic curves over $\Fq$ with $q + 1 + m - k$ 
points, where $k$ is as in the statement of the Expectation.  Let $t = k - m$; 
it is easy to check that when $q$ is prime, $t$ is coprime to $q$, so the 
number of elliptic curves of trace $t$ is then equal to the Kronecker class
number $H(\Delta)$ of the discriminant $\Delta = t^2 - 4q$.  If the Generalized
Riemann Hypothesis is true, then up to factors of $\log\log{\abs{\Delta}}$, this 
class number is bounded below and above by $\sqrt{\abs{\Delta}}$.  (For fundamental
discriminants, this is~\cite{Littlewood1927}*{Theorem~1, p.~367}, and the 
result for general discriminants follows easily.)

If we write $m = 2\sqrt{q} - \eps$ with $0\le\eps < 1$, then
$t = k+\eps - 2\sqrt{q}$ and
\[
\Delta = (k + \eps)^2 - 4(k + \eps)\sqrt{q},
\]
so certainly $12\sqrt{q} > \abs{\Delta} > \sqrt{q}$.  Therefore, assuming GRH,
we expect that up to factors of $\log \log q$, the number of elliptic curves 
with defect $k$ is bounded below and above by $q^{1/4}$.

Now consider applying Algorithm~\ref{A:construction} to the irreducible cubics 
$f_1$ and $f_2$, taking the lists $\calL_1$ and $\calL_2$ to both be the set of 
elliptic curves of defect $k$.  (Note that the value of $k$ is chosen so that
the curves of defect $k$ have no rational $2$-torsion points, so they are 
compatible with $f_1$ and $f_2$.)

According to Heuristic Expectation~\ref{HE:success}, we expect the algorithm
to succeed with probability $(9/4) (\#\calL_1)^2 / q$.  We have just seen that
$\#\calL_1\sim q^{1/4}$, up to factors of $\log\log q$; therefore, we expect
success with probability $q^{-1/2}$, up to factors of $\log\log q$.  If we have
success, then the resulting curve $D_a$ will satisfy
\[ d(D_a) = d(C) + d(E_{a,1}) + d(E_{a,2}) = d(C) + 2k. \qedhere\]
\end{proof}

We expect a similar result for cubic polynomials $f_1$ and $f_2$ with other 
factorizations.  We will only explicitly state one.

\begin{expectation}
\label{HE:singlerootcubics}
Let $C$ be a genus-$2$ curve over a finite prime field~$\Fq$ that can be 
written in the form $y^2 = f_1 f_2$, where $f_1$ and $f_2$ are cubic 
polynomials in~$\Fq[x]$ each with exactly one rational root. Let
$m = \lfloor 2\sqrt{q}\rfloor$.  If $m$ is even set $k = 2$\textup{;} if $m$ is 
odd set $k = 1$. Up to powers of $\log \log q$, the probability that there is a 
double cover $D$ of $C$, of the type described in 
Section~\textup{\ref{S:family}}, such that the defect $d(D)$ of $D$ satisfies
\[d(D) \le d(C) + 2k\]
is approximately $q^{-1/2}$.
\end{expectation}

(Note that, compared to Heuristic Expectation~\ref{HE:norootcubics}, the values
of $k$ are assigned in the opposite way.)

The justification for this expectation is essentially the same as that for 
Heuristic Expectation~\ref{HE:norootcubics}.  The main difference is that now 
we are considering elliptic curves with even traces and hence even group orders,
but we do not want our sets $\calL_1$ and $\calL_2$ to include elliptic curves 
that have all of their $2$-torsion defined over the base field.  Fortunately, 
it is not hard to see (using, for instance, the theory of isogeny 
volcanoes~\cite{FouquetMorain2002}) that at least half of the curves of trace 
$t$ do not have all of their $2$-torsion points defined over the base field.

\section{Interlude on work by Hayashida}
\label{S:Hayashida}

In the late 1960s, Hayashida and Nishi studied genus-$2$ curves lying on
products of isogenous elliptic curves with complex multiplication. Their 
initial paper~\cite{HayashidaNishi1965} studied the general case, and was 
followed by a paper by Hayashida~\cite{Hayashida1968} that considered the 
special case of curves lying on $E\times E$, for $E$ with CM by a maximal 
order. In this section we will reinterpret Hayashida's work in terms of an
equivalence of categories that is reminiscent of both 
\begin{enumerate}
\item the equivalence of categories between supersingular elliptic curves
      and rank-$1$ right modules over a maximal order in a quaternion
      algebra~\cite{Kohel1996}, and
\item the bijection between supersingular abelian surfaces (given with
      an action of a maximal order in a quaternion algebra) and ``oriented
      maximal orders''~\cite{Ribet1989}*{\S 3}.
\end{enumerate}

First we will define a category of principal polarizations on the square of an 
elliptic curve.  Then we will show that this category is equivalent to the 
category of rank-$1$ right modules over an order in a certain quaternion 
algebra. Finally, we will show that there is an involution on the category such
that the orbits of isomorphism classes of objects correspond to ``good curves'' 
of genus~$2$ whose (unpolarized) Jacobian varieties are isomorphic to the 
square of the given elliptic curve.

\subsection{The category of principal polarizations on 
            \texorpdfstring{$E\times E$}{E x E}}
Let $E$ be an elliptic curve over an arbitrary field $k$, let $\calO$ be the
ring of ($k$-rational) endomorphisms of $E$, and suppose that $\calO$ is
isomorphic to the ring of integers of an imaginary quadratic field $K$ of
discriminant~$\Delta$.  Let $A$ be the abelian surface $E\times E$, and note
that the principal polarizations on $A$ are in bijection with the set
of positive definite unimodular Hermitian matrices in the matrix
ring~$M_2(\calO)$.  We denote complex conjugation (in $\calO$ and in $K$)
by $x\mapsto \bar{x}$.

Let $P\mapsto P'$ be the involution on $M_2(\calO)$ that sends a matrix
$P = \left[\begin{smallmatrix}a&b\\c&d\end{smallmatrix}\right]$ to
\[ P' = \begin{bmatrix}\phantom{-}\bar{d}&-\bar{c}\\-\bar{b}&\phantom{-}\bar{a}\end{bmatrix}.\]
Let $P\mapsto P^*$ denote the conjugate-transpose involution, and note that
$P^* P' = (\bar{\det P}) I$, where $I$ is the identity matrix.

We define a category $\pol E^2$ as follows: The objects of $\pol E^2$ are
positive definite unimodular Hermitian matrices in $M_2(\calO)$, that is,
principal polarizations on $E\times E$.  The set of morphisms from one object
$L$ to another object $M$ is defined to be
\[ \Hom(L,M) = \{ P\in M_2(\calO) \colon MP = P'L\}, \]
and composition of morphisms
\[ \Hom(M,N)\times \Hom(L,M) \to \Hom(L,N) \]
is given by sending $(Q,P)$ to $QP$.  It follows that $\pol E^2$ is a
preadditive category.  Note that if $P$ is a morphism from $L$ to $M$, and if 
we multiply both sides of the equality $MP = P'L$ by $P^*$, we find that 
\[ P^* M P = (\bar{\det P}) L. \]
Since $M$ and $L$ are both positive definite, we see that $\det P$ must be a
non-negative rational integer.

Let $I$ denote the identity matrix in $M_2(\calO)$.  We compute that
\[
\End I = \left\{ \begin{bmatrix}\alpha&-\bar{\beta}\\ 
                                \beta&\phantom{-}\bar{\alpha}\end{bmatrix}
                                \colon \alpha,\beta\in\calO \right\}. 
\]

\begin{theorem}
\label{T:EndI}
The ring $\End I$ is an order in a quaternion algebra 
$\HH = (\End I)\otimes\QQ$ over $\QQ$.  The algebra $\HH$ is ramified at
infinity, at the prime divisors of $\Delta$ that are congruent to $3$ 
modulo~$4$, and at $2$ if $\Delta$ is even but not congruent to $8$ 
modulo~$32$.  The reduced discriminant of $\End I$ is equal to the 
discriminant of $\calO$.
\end{theorem}

\begin{proof}
Set
\[ i = \begin{bmatrix}0 & -1\\ 1 & \phantom{-}0\end{bmatrix}\in \End I,\]
so that $i^2 = -1$.  Let $\calObar$ denote the image of $\calO$ in 
$M_2(\calO)$ under the twisted embedding
\[
a \mapsto \begin{bmatrix}a & 0 \\ 0 & \bar{a}\end{bmatrix}.
\]
Then $\End I = \calObar[i]$.  We write elements of $\calObar[i]$ as 
$\alpha + i\beta$, with $\alpha,\beta\in\calO$, and we note that 
$\alpha i = i\bar{\alpha}$ for all $\alpha\in\calO$.  Likewise, we let $\Kbar$
denote the twisted diagonal image of $K$ in $M_2(K)$, and we note that 
$\HH = \Kbar[i]$.  Clearly $\HH$ is a quaternion algebra, and clearly $\HH$ is
ramified at infinity.

To find the finite primes that ramify in $\HH$, we note that for all
$\alpha,\beta\in K$ we have
\begin{align*}
(\alpha + i\beta) (\bar{\alpha} - i\beta)
&=
\alpha\bar{\alpha} + i\beta\bar{\alpha} - \alpha i\beta - i\beta i \beta\\
&=
\alpha\bar{\alpha} + i\beta\bar{\alpha} - i\bar{\alpha}\beta - i^2\bar{\beta}\beta\\
&= 
N(\alpha) + N(\beta),
\end{align*}
where $N$ denotes the norm from $K$ to $\QQ$, so a nonzero $\alpha + i\beta$
has norm $0$ if and only if $N(\beta) = -N(\alpha)$.  Setting 
$\gamma = \beta/\alpha$, we see that $\HH$ is unramified at a prime $p$ if and 
only if there is an element $\gamma\in K_p = K\otimes\QQ_p$ such that
$\gamma\bar{\gamma} = -1$.

The norm map from $K_p^*$ to $\QQ_p^*$ is surjective on units for all primes 
$p$ that are unramified in $\calO$, so $\HH$ is unramified at these primes.  
We are left to consider the primes that are ramified in $\calO$; let $p$ be
such a prime.

If $p\equiv 1 \bmod 4$ then $\ZZ_p\subset\calO_p$ contains an element $e$ whose
square is $-1$ and that is fixed by complex conjugation, so $p$ is unramified 
in $\HH$.  On the other hand, if $p\equiv 3\bmod 4$ then no element of $K_p$ 
has norm $-1$, so $p$ is ramified in~$\HH$.

This leaves us with the case $p=2$.  We know that $\QQ_2^*/\QQ_2^{*2}$ is a
group of order~$8$, generated by the images of $-1$, $2$, and~$3$, and we 
calculate the following table:
\begin{center}
\begin{tabular}{rrccc}
\toprule
     && Description                 && Is $-1$ a norm in\\
$ n$ &&  of $\QQ_2(\sqrt{n})/\QQ_2$ && this extension?\\
\midrule
$ 1$ &&      split && yes \\
$-1$ &&   ramified &&  no \\
$ 2$ &&   ramified && yes \\
$-2$ &&   ramified &&  no \\
$ 3$ &&   ramified &&  no \\
$-3$ && unramified && yes \\
$ 6$ &&   ramified &&  no \\
$-6$ &&   ramified && yes \\
\bottomrule
\end{tabular}
\vskip 3ex
\end{center}
We see that $\HH$ is ramified at $2$ if and only if the $2$-adic extension
$K_2$ of $\QQ_2$ is isomorphic to $\QQ_2(\sqrt{n})$ with $n = -1$, $-2$, $3$,
or $6$. Summarizing, we see that if $\HH$ is ramified at $2$ then so is $K$,
and if $K$ is ramified at $2$ then so is $\HH$, unless the $2$-adic extension 
is isomorphic to $\QQ_2(\sqrt{n})$ with $n = 2 $ or $n = -6$.  This can be 
summarized even more briefly: if $\Delta$ is even, then $\HH$ is ramified at
$2$ unless $\Delta \equiv 8\bmod 32$.

Finally, a direct computation shows that the trace dual of $\calObar[i]$ is 
$\calAbar[i]$, where $\calA$ is the trace dual of $\calO$.  It follows that the
reduced discriminant of $\calObar[i]$ is $\Delta$.
\end{proof}

\subsection{An equivalence of categories}
We define a functor $F$ from the category $\pol E^2$ to the category of 
rank-$1$ projective right $\calObar[i]$-modules as follows: If $L$ is a 
positive definite unimodular matrix, we take $F(L)$ to be the right 
$(\End I)$-module $\Hom(I,L)$. If $P$ is a morphism from $L$ to $M$, we take 
$F(P)$ to be the morphism of right $\calObar[i]$-modules that takes an element 
$Q$ of $F(L) = \Hom(I,L)$ to the element $PQ$ of $F(M) = \Hom(I,M)$.

\begin{theorem}
\label{T:equiv}
Suppose $\Delta\not\equiv 0\bmod 8.$  Then the functor $F$ is an equivalence 
of categories.
\end{theorem}

\begin{proof}
Before we begin the proof proper, we describe the module $\Hom(I,L)$ a little
more concretely.  Suppose that 
$L = \begin{bmatrix} k & \alpha \\ \bar{\alpha} & \ell\end{bmatrix}$ is
an element of
$M_2(\calO)$.  Let $S = \begin{bmatrix} k & \alpha \\ 0 & 1\end{bmatrix}$,
so that $S^*S = (\det S) L$.  By definition,
\[
\Hom(I,L) = \{ P\in M_2(\calO) \colon L P = P' \}.
\]
The condition $LP = P'$ translates into $S^* S P = (\det S) P'$, which is 
equivalent to $SP = (SP)'$.  Since the set of elements $x \in M_2(K)$ 
satisfying $x = x'$ is exactly $\bar{K}[i]$, this means that
\[
\Hom(I, L) = M_2(\calO) \cap S^{-1}\cdot \bar{K}[i],
\]
where the intersection takes place in $M_2(K)$.

As a right $\calObar[i]$-module, $\Hom(I,L)$ is isomorphic to 
\[
S\Hom(I,L) = \left\{ u + iv \in \bar{K}[i] :
                    \begin{bmatrix} 1 & -\alpha\\ 0 & \phantom{-}k\end{bmatrix}
                    \begin{bmatrix} u & -\bar{v}\\ v & \phantom{-}\bar{u}\end{bmatrix}
                    \in k M_2(\calO)
                    \right\}.
\]
It is easy to check that this set is generated (as a right $\calO$-module)
by the two elements $k$ and $\alpha + i$.  Thus, $\Hom(I,L)$ is isomorphic as a
right $\calObar[i]$-module to the ideal $(k, \alpha + i)$. In particular, 
$\Hom(I,L)$ is a projective rank-$1$ $\calObar[i]$-module.

To prove that $F$ gives an equivalence of categories, it suffices to show that
$F$ is fully faithful and that $F$ is surjective on isomorphism classes of 
objects.

\emph{Surjectivity}:
Section 3 of~\cite{Hayashida1968} (see in particular the final paragraph) shows
that there is a bijection between right $\calObar[i]$-ideals and what Hayashida
calls ``proper classes'' of unimodular Hermitian matrices.  One can check that
two unimodular Hermitian matrices are in the same proper class (according to 
Hayashida's definition~\cite{Hayashida1968}*{p.~31}) if and only if there are
isomorphic in the category $\pol E^2$.  Thus, Hayashida already gives us a 
bijection between isomorphism classes of objects in $\pol E^2$ and equivalence
classes of right $\calObar[i]$-ideals.

\emph{Full fidelity}:
Suppose $L$ and $M$ are two objects in $\pol E^2$, say
\[ L = \begin{bmatrix} k & \alpha \\ \bar{\alpha} & \ell\end{bmatrix} 
\text{\quad and\quad}
M = \begin{bmatrix} m & \beta \\ \bar{\beta} & n\end{bmatrix}.\]
Let
\[ S = \begin{bmatrix} k & \alpha \\ 0 & 1\end{bmatrix} 
\text{\quad and\quad}
T = \begin{bmatrix} m & \beta \\ 0 & 1\end{bmatrix},\]
so that $S^*S = (\det S) L$ and $T^* T = (\det T) M$. By definition, 
$\Hom(L,M)$ is equal to the set of $P\in M_2(\calO)$ with $MP = P'L$, and a 
calculation shows that this last condition is equivalent to  
$TPS^{-1} = (TPS^{-1})'$.  Thus,
\[
\Hom(L,M) = M_2(\calO) \cap T^{-1}\cdot \bar{K}[i] \cdot S.
\]
If $P$ is an element of $\Hom(L,M)$, then the homomorphism $F(P)$ from 
\[
\Hom(I,L) = M_2(\calO) \cap S^{-1}\cdot \bar{K}[i]
\]
to
\[
\Hom(I,M) = M_2(\calO) \cap T^{-1}\cdot \bar{K}[i]
\]
is simply multiplication on the left by~$P$.

Now, every morphism from a right $\calObar[i]$-ideal to another right 
$\calObar[i]$-ideal is obtained from left multiplication by an element of
$\bar{K}[i]$.  In particular, every element of $\Hom(S F(L) ,T F(M))$ is given 
by left multiplication by an element of $\bar{K}[i]$, so every element of 
$\Hom(F(L),F(M)$ is given by left multiplication by an element $P$ of
$T^{-1}\cdot \bar{K}[i]\cdot S$.  Multiplication by $P$ will take the lattice
$M_2(\calO)\subset S^{-1}\cdot \bar{K}[i]$ to the lattice 
$M_2(\calO)\subseteq T^{-1}\cdot\bar{K}[i]$ if and only if $P$ lies in 
$M_2(\calO)$.  Thus, $\Hom(F(L),F(M))$ is also isomorphic to 
$M_2(\calO)\cap T^{-1}\cdot \bar{K}[i] \cdot S$, and the functor $F$ gives one
such isomorphism.
\end{proof}

\subsection{An involution on the category of principal polarizations}
Set
\[ s = \begin{bmatrix}1 & \phantom{-}0\\ 0 & -1\end{bmatrix}\in M_2(\calO),\]
so that $s^2 = I$ and $s' = -s$. We define an involution on $\pol E^2$ as 
follows: If $L$ is an object of $\pol E^2$ we set $\bar{L} = sLs$, and we 
define an isomorphism $\Hom(L,M)\to \Hom(\bar{L},\bar{M})$ by sending $
P\in \Hom(L,M)$ to $\bar{P} = s P s$.  In this subsection we will define a 
second category, $\cur E^2$, whose isomorphism classes of objects correspond to
orbits of isomorphism classes of objects of $\pol E^2$ under this involution.

The objects of $\cur E^2$ are the same as the objects of $\pol E^2$: namely, 
positive definite unimodular Hermitian matrices in $M_2(\calO)$, which we can 
also consider to be principal polarizations on $E\times E$.  The set of 
morphisms from one object $L$ to another object $M$ is defined to be
\[ \Hom(L,M) = \{ P\in M_2(\calO) \colon MP = \pm P'L\}; \]
note the plus-or-minus sign in the definition. As in $\pol E^2$, composition
of morphisms 
\[ \Hom(M,N)\times \Hom(L,M) \to \Hom(L,N) \]
is given by sending $(Q,P)$ to $QP$. Unlike $\pol E^2$, the category $\cur E^2$
is not preadditive, because if $P_1$ is an endomorphism of $L$ that takes the 
plus sign in the definition, and if $P_2$ is an endomorphism that takes the 
minus sign, then in general $P_1 + P_2$ will not be an endomorphism of~$L$.

Note that if $P$ is a morphism from $L$ to $M$, and if we multiply both sides 
of the equality $MP = \pm P'L$ by $P^*$, we find that  
\[ P^* M P = \pm (\bar{\det P}) L. \]
Since $M$ and $L$ are both positive definite, we see that $\det P$ is a 
rational integer.  Thus, a morphism from $L$ to $M$ is an element $P$ of 
$M_2(\calO)$ such that $\det P$ is a rational integer and such that 
\[ P^* M P = \lvert\det P\rvert L. \]

Recall that a \emph{good curve of genus $2$} is either a nonsingular curve of
genus~$2$ or a pair of elliptic curves crossing transversely at a point.
(Over a field that is not algebraically closed, the two elliptic curves may be 
a Galois conjugate pair of curves defined over a quadratic extension.)
For the rest of this section, we will simply write \emph{good curve} when we 
mean a good curve of genus~$2$. If $C$ is a good curve consisting of two 
elliptic curves crossing transversely at a point, its Jacobian is the product
of the two elliptic curves, together with the product principal polarization.
Torelli's theorem in genus $2$ says that the Jacobian map from the set of good 
curves to the set of principally-polarized abelian surfaces is a bijection.

Suppose $C_1$ and $C_2$ are good curves whose Jacobians are isomorphic
to $E^2$ (as unpolarized surfaces).  Let $\varphi_1$ and $\varphi_2$ be
isomorphisms from $\Jac C_1$ and $\Jac C_2$ to~$E^2$, let $\lambda_1$ and 
$\lambda_2$ be the canonical polarizations on $\Jac C_1$ and $\Jac C_2$,
let $\mu$ be the product principal polarization on~$E^2$, and for each $i$
let $L_i = \mu^{-1}\varphi_{i*}\lambda_i$, so that $L_i\in M_2(\calO)$ is
a positive definite unimodular Hermitian form.

Note that the good curves $C_1$ and $C_2$ are isomorphic to one another
if and only if their polarized Jacobians are isomorphic to one another,
which will be the case if and only if there is an invertible
$P\in M_2(\calO)$ such that $P^* L_1 P =  L_2$.  If the discriminant
$\Delta$ of $\calO$ is anything other than $-3$ or~$-4$, then
such an invertible $P$ must have determinant $\pm 1$, in which case $P$
also gives an isomorphism between $L_1$ and $L_2$ in the category
$\cur E^2$.  Thus, when $\lvert\Delta\rvert > 4$, the objects of $\cur E^2$
can be viewed as the good curves over $k$ whose Jacobians are isomorphic 
to~$E^2$.

\section{Genus-\texorpdfstring{$2$}{2} curves with small defect}
\label{S:genus2}

Given Heuristic Expectations~\ref{HE:norootcubics} and~\ref{HE:singlerootcubics},
our strategy for producing genus-$4$ curves with small defect is clear: We 
should try to produce a large number of small-defect curves of genus~$2$ that
can be written $y^2 = f_1 f_2$ for cubic polynomials $f_1$ and $f_2$, and then 
apply Algorithm~\ref{A:construction} to  all of the pairs $(f_1, f_2)$, taking 
the sets $\calL_1$ and $\calL_2$ to be the elliptic curves of small defect. As
long as we have significantly more that $q^{1/2}$ curves of genus~$2$ to work
with, we should find a small-defect curve of genus~$4$ in this way.

In this section, we show that in some cases we can prove that there are
sufficiently many genus-$2$ curves of small defect, and we have an efficient
way of producing them.  

\begin{theorem}
\label{T:genus2min}
Let $q$ be a prime power and let $t$ be an integer, coprime to $q$, with 
$\abs{t}\le \lfloor2\sqrt{q}\rfloor$.  Let $\Delta = t^2 - 4q$, write 
$\Delta = \Delta_0 F^2$ for a fundamental discriminant~$\Delta_0$, and let $r$
be the number of prime factors of $\Delta_0$. If $\abs{\Delta_0} > 4$ then the
number of genus-$2$ curves over $\Fq$ with Weil polynomial $(x^2 - tx + q)^2$ 
is at least
\[
\frac{h(\Delta_0) \varphi(\abs{\Delta_0})}{12 \cdot 2^r},
\]
where $\varphi$ is the Euler $\varphi$-function.
\end{theorem}

In fact, in the case where $\abs{\Delta}$ is prime, we have an exact value for the
number of genus-$2$ curves with the specified Weil polynomial.

\begin{theorem}
\label{T:genus2exact}
Let $q$ be a prime power and let $t$ be an integer, coprime to~$q$, with 
$\abs{t}\le \lfloor2\sqrt{q}\rfloor$.  Let $\Delta = t^2 - 4q$. If 
$\abs{\Delta}$ is a prime greater than~$3$, then the number of genus-$2$ curves
over $\Fq$ with Weil polynomial $(x^2 - tx + q)^2$ is exactly $N h(\Delta)$,
where
\begin{equation}
\label{EQ:Hayashida}
N = \left\lceil \frac{-\Delta}{24} \right\rceil + \frac{h(\Delta)-1}{2}.
\end{equation}
\end{theorem}

\begin{proof}[Proof of Theorems~\textup{\ref{T:genus2min}} 
              and~\textup{\ref{T:genus2exact}}]
Suppose we are in the situation of Theorem~\textup{\ref{T:genus2min}}, and let 
$E$ be an elliptic curve over $\Fq$ with trace $t$ and with endomorphism ring
of discriminant $\Delta_0$, that is, with endomorphism ring a maximal order.  
Hayashida~\cite{Hayashida1968}*{pp.~42--43} gives an exact formula for the 
number $N$ of nonsingular genus-$2$ curves whose Jacobians are isomorphic 
(as unpolarized varieties) to the product $E\times E$; he works over the 
complex numbers, but the argument works for ordinary elliptic curves over 
finite fields as well.\footnote{
    Note that there is a misprint on page 43 of~\cite{Hayashida1968}:
    The term $(1/4)(1 - (-1))^{(m^2-1)/8}$ in the second line should be
    $(1/4)(1 - (-1)^{(m^2-1)/8}) h$.}
The number $N$ depends on $\Delta$ in a somewhat complicated way, but for our
purposes we need only note two facts.  First, if $\abs{\Delta_0} > 4$ then $N$
is at least $\varphi(\abs{\Delta_0}) / 24$, and second, when $\abs{\Delta}$ is
greater than~$3$ and is prime (and hence $3 \mod 8$), Hayashida's formula for
$N$ reduces to~\eqref{EQ:Hayashida}.

How many abelian surfaces with Weil polynomial $(x^2 - tx + q)^2$ can be 
written as $E \times E$, where $E$ has CM by a maximal order? The set of $E$ 
with CM by $\Delta_0$ is in bijection with the class group of $\Delta_0$, and
two such curves $E_1$ and $E_2$ have isomorphic squares if and only if their
associated ideal classes have squares that are equal.  Thus, the set of 
surfaces of the form $E\times E$ where $E$ has CM by $\Delta_0$ is in bijection 
with the set of squares in the class group. Since the $2$-rank of the class
group is~$r-1$, there are $h(\Delta_0)/2^{r-1}$ such surfaces.  Combining this
equality with the lower bound on $N$ from the proceeding paragraph
gives us the lower bound of Theorem~\ref{T:genus2min}.

When $-\Delta$ is prime we have $\Delta_0 = \Delta$, and every element of the 
class group is a square.  Thus, there are exactly $N h(\Delta)$ curves over 
$\Fq$ with the given Weil polynomial, as claimed in Theorem~\ref{T:genus2exact}.
\end{proof}

Theorems~\ref{T:genus2min} and~\ref{T:genus2exact} include a requirement that
$t$ be coprime to~$q$.  We will see similar conditions frequently enough in 
what follows to justify the following definition.

\begin{definition}
\label{D:exceptional}
Let $d$ be a non-negative integer.  A prime power $q$ is \emph{$d$-exceptional} 
if there are no elliptic curves of defect $d$ over the finite field $\Fq$; 
otherwise, $q$ is \emph{$d$-unexceptional}.
\end{definition}

For positive $d$ and for $q$ that are not too small, it is easy to tell when 
$q$ is $d$-exceptional.

\begin{proposition}
Let $d$ be a positive integer and let $q$ be a prime power with $q > 56 d^2$.
Then $q$ is $d$-exceptional if and only if $q$ is not coprime to 
$\lfloor 2 \sqrt{q} \rfloor - d$.
\end{proposition}

\begin{proof}
Let $t = \lfloor 2\sqrt{q} \rfloor - d$. If $q > 56 d^2$ then $t$ certainly
lies in the Weil interval, and according 
to~\cite{Schoof1987}*{Theorem~4.2, p.~193} there will exist elliptic curves 
over $\Fq$ with trace $t$ if and only if either $t$ is coprime to $q$ or $t$ 
is not coprime to $q$ and lies in a short list of possible values.  The defects
associated to these possible values are $0$ (if $q$ is a square) and possibly 
several  other values, all of which are at least $(2 - \sqrt{3})\sqrt{q} - 1$. 
Since $q > 56 d^2$, if $q$ is $d$-unexceptional then $d$ is one of these
values and we have
\[d \ge (2 - \sqrt{3})\sqrt{q} - 1 > (2 - \sqrt{3})\sqrt{56} \,d  - 1> 2d - 1,\]
which is impossible.  Therefore, there are elliptic curves over $\Fq$ of defect
$d$ if and only if $\lfloor \sqrt{2q} \rfloor - d$ is coprime to $q$.
\end{proof}

\begin{corollary}
Let $d$ be a positive integer and let $q$ be a prime power with $q > 56 d^2$.
If $q$ is prime, or if $q$ is a square that is coprime to~$d$, then $q$ is 
$d$-unexceptional.\qed
\end{corollary}

Heuristically, for every fixed $d>0$ one expects the number of $d$-exceptional 
primes powers less than $x$ to grow like a constant times $\log x$.

In Theorem~\ref{T:genus2min}, if $q$ is coprime to $t$ and if the conductor $F$
is equal to~$1$, then the theorem leads to a lower bound of
\[
\frac{h(\Delta) \varphi(\abs{\Delta})}{12 \cdot 2^r}
\ge 
\frac{h(\Delta) \varphi(\abs{\Delta})}{12 \cdot \calD(\abs{\Delta})}
\]
for the number of genus-$2$ curves with the specified Weil polynomial, where 
$\calD$ is the divisor function. We know that there is a constant $c$ such that 
$\varphi(n) > c n /\log n$ for all $n$ 
(see \cite{HardyWright1938}*{Theorem~328, p.~267}), and we know that under the 
GRH we have $h(\Delta) > c'\sqrt{\abs{\Delta}} / \log\log{\abs{\Delta}}$ for some 
constant $c'$ (see \cite{Littlewood1927}*{Theorem~1, p.~367}). Furthermore, the 
divisor function $\calD(n)$ has average order $\log n$ and normal order 
$(\log n)^{\log 2}$ (see \cite{HardyWright1938}*{Theorem~319, p.~264} and 
\cite{HardyWright1938}*{Theorem~432, p.~359}). In the case where we are 
considering the squares of isogeny classes of elliptic curves with small 
positive defect~$d$, we have $\Delta\approx d \sqrt{q}$, and we are led to 
suggest the following heuristic.

\begin{heuristic}
\label{H:smalldefect2}
For every prime $q$ and integer $d\ge 0$, let $M_{q,d}$ be the number of 
genus-$2$ curves over $\Fq$ with Jacobians isogenous to the square of an 
elliptic curve with defect~$d$. For each fixed $d > 0$, we will model $M_{q,d}$
as growing like $q^{3/4}$, up to logarithmic factors, as $q$ ranges over the 
prime powers that are $d$-unexceptional.
\end{heuristic}

Note that when $d=0$ we would still expect $M_{q,d}$ to be $\Otilde(q^{3/4})$,
but we do not expect a lower bound of the same shape --- indeed, if $q$ is a
prime of the form $n^2 + 1$, then $M_{q,0} = 0$.

Heuristic~\ref{H:smalldefect2} suggests that there are relatively many 
genus-$2$ curves with small defect. For the purpose of constructing examples,
though, we need more than a simple statement of existence; we would like a way
of \emph{producing} these curves of small defect. Here is an algorithm that in 
certain cases is guaranteed to do so.

\begin{algorithm}
\label{A:genus2}
\begin{algtop}
\algin  An odd prime power~$q$ and a list $\calL$ of elliptic curves
        over $\Fq$ with defect at most~$d$.
\algout A list of genus-$2$ curves over $\Fq$ with defect at most~$2d$.
\end{algtop}
\begin{alglist}
\item Initialize $S$ to be the empty list.
\item \label{genus2-step2}
      For every pair of elliptic curves $E_1, E_2$ in $\calL$:
\begin{algsublist}
\item Use Algorithm~5.1 (pp.~183--184) of~\cite{BrokerHoweEtAl2015} to compute
      the set of genus-$2$ curves over $\Fq$
      whose Jacobians are $(2,2)$-isogenous
      to $E_1\times E_2$.
\item Append to $S$ all of the resulting
      curves that are not already isomorphic to a curve in $S$.
\end{algsublist}      
\item Set $i = 1$.
\item \label{genus2-step4}
      Repeat the following until $i > \#S$:
\begin{algsublist}
\item Let $C$ be the $i$-th element of $S$.
\item Compute all of the genus-$2$ curves over $\Fq$
      that are Richelot isogenous to~$C$, using the formulas
      from~\cite{BruinDoerksen2011}*{\S 4}.
\item Append to $S$ all of the resulting
      curves that are not already isomorphic to a curve in $S$.
\item Increment $i$.      
\end{algsublist}
\item Output $S$.
\end{alglist}
\end{algorithm}

\begin{theorem}
\label{T:whenitworks}
Let $q$ be an odd prime power, let $d\ge0$ be an integer, let 
$t = \lfloor 2\sqrt{q}\rfloor - d$, and suppose $t$ is odd, positive, and 
coprime to~$q$.  Let $\Delta = t^2 - 4q$, and write $\Delta = \Delta_0 F^2$ 
for a fundamental discriminant $\Delta_0$. Let $S$ be the list produced by 
Algorithm~\textup{\ref{A:genus2}} given $q$ and the list of defect-$d$ 
elliptic curves over~$\Fq$ as input. Suppose that the exponent of the class 
group of $\Delta_0$ is greater than~$2$. Then $S$ contains all genus-$2$
curves $C$ over $\Fq$ such that the Jacobian of $C$ is isomorphic 
{\upleftparen}as an unpolarized abelian surface{\uprightparen} to the square
of an elliptic curve with defect $d$ and with CM by~$\Delta_0$.
\end{theorem}

Note that there are at most $66$ negative fundamental discriminants whose class
groups have exponent at most $2$, and there are exactly $65$ such discriminants
if the GRH is true, the largest in absolute value being 
$-5460$~\cite{Weinberger1973}*{Theorem~1, p.~119}.  A list of these $65$ 
discriminants is given in~\cite{BorevichShafarevich1966}*{Table~5.1, p.~426}.

\begin{proof}[Proof of Theorem~\textup{\ref{T:whenitworks}}]
Recall that a Richelot isogeny from a genus-$2$ curve $C$ is obtained by taking
a subgroup-scheme $G$ of $(\Jac C)[2]$ that is maximal isotropic with respect
to the Weil pairing and observing that the quotient variety $A = (\Jac C)/G$ 
has a natural principal polarization $\lambda$ such that the pullback of 
$\lambda$ to $\Jac C$ is twice the canonical principal polarization on~$\Jac C$.
When the principally-polarized surface $(A,\lambda)$ is the Jacobian of a curve
$D$, we say that we have a Richelot isogeny from $C$ to $D$ (or from $\Jac C$
to $\Jac D$).  But $(A,\lambda)$ might also \emph{not} be a Jacobian; it might
be the product of two elliptic curves with the product polarization.\footnote{
     \emph{A priori}, it might also be the Weil restriction of a
     polarized elliptic curve over
     $\FF_{q^2}$~\cite{GonzalezGuardiaEtAl2005}*{Theorem~3.1, p.~270},
     but no such Weil restriction is isogenous to the square of an
     elliptic curve with nonzero trace.}
This is precisely the situation discussed in~\cite{HoweLeprevostEtAl2000}*{\S 3}.
One could continue to say $(A,\lambda)$ is the Jacobian of a curve --- the
\emph{singular} genus-$2$ curve consisting of the union of the two elliptic 
curves, crossing transversely at their origins.  We will use the term 
``generalized Richelot isogeny'' to refer to this slightly expanded concept; 
however, we should keep in mind that Algorithm~\ref{A:genus2} refers only to 
Richelot isogenies between nonsingular genus-$2$ curves. 

Let $\calO$ be the imaginary quadratic order of discriminant $\Delta_0$. We
note that $\Delta \equiv 5\bmod 8$, so we also have $\Delta_0\equiv 5 \bmod 8$,
so that $2$ is inert in $\calO$, and the only simple $\calO$-module of
$2$-power order is $\calO/2\calO$. Let $E$ be an elliptic curve with trace $t$ 
and with CM by~$\calO$. The elliptic curves isogenous to $E$ with CM by $\calO$
form a principal homogeneous space for the class group $\Cl\calO$ of $\calO$; 
we choose $E$ to be a base point for the action of the class group. Finitely 
generated torsion-free $\calO$-modules are determined by their rank and their
Steinitz class; as a consequence, if $E_1$, $E_2$, $E_3$, and $E_4$ are 
elliptic curves corresponding to elements $g_1$, $g_2$, $g_3$, and $g_4$ of
$\Cl\calO$, then the surfaces $E_1 \times E_2$ and $E_3 \times E_4$ are 
isomorphic if and only if $g_1 + g_2 = g_3 + g_4$. If $A$ is an abelian surface
isogenous to $E\times E$ with Frobenius endomorphism $\pi$ such that 
$\QQ(\pi)\cap\End A \cong \calO$, then $A$ is isomorphic to $E_1 \times E_2$
for two elliptic curves with CM by $\calO$, which themselves correspond as 
above to two element $g_1$ and $g_2$ of the class group~$\Cl\calO$; the 
\emph{Steinitz class} of $A$ (with respect to the base point $E$) is the 
element $g_1+g_2$ of~$\Cl\calO$.

Let $g$ be an element of $\Cl\calO$ that is not $2$-torsion, and let $E_1$ and 
$E_2$ be the elliptic curves corresponding to $g$ and to~$-g$, so that $E_1$ 
and $E_2$ are not isomorphic to one another. Let $P_1\in E_1[2](\Fqbar)$ and
$P_2\in E_2[2](\Fqbar)$ be generators for the simple $\calO$-modules
$E_1[2](\Fqbar)$ and $E_2[2](\Fqbar)$, and let 
$\psi\colon E_1[2](\Fqbar) \to E_2[2](\Fqbar)$ be the unique $\calO$-module
isomorphism that sends $P_1$ to $P_2$.  Then the construction 
of~\cite{HoweLeprevostEtAl2000}*{Proposition~4, p.~324}, applied to $E_1$, 
$E_2$, and $\psi$, will produce a genus-$2$ curve $C$ whose Jacobian $J$ is 
isomorphic to $E_1\times E_2$ divided by the graph $X$ of $\psi$.  The ring 
$\calO$ acts on $X$, so $\calO$ acts on $J$ compatibly with Frobenius, so $J$ 
has a Steinitz class. The class of the $\calO$-module $X$ in the class group is
trivial, so the Steinitz class of $J$ is equal to that of $E_1\times E_2$,
which is trivial; therefore $J$ is isomorphic to $E^2$. The curve $C$ is 
included in the output of Algorithm~5.1  of~\cite{BrokerHoweEtAl2015}, applied 
to $E_1$ and~$E_2$. This shows that after step~\ref{genus2-step2}, the set $S$
from Algorithm~\ref{A:genus2} includes at least one curve $C$ whose Jacobian is 
isomorphic (as an unpolarized surface) to $E^2$.  Let $D$ be any other curve 
over $\Fq$ whose Jacobian is isomorphic to $E^2$.  We will show that there is 
a sequence of generalized Richelot isogenies starting from $\Jac C$ and ending 
at $\Jac D$. To prove this, we will work with the category $\pol E^2$ discussed 
in Section~\ref{S:Hayashida}. 

The quaternion algebra $\HH$ associated in Section~\ref{S:Hayashida} to 
$\pol E^2$ is the quaternion algebra over $\QQ$ ramified at $\infty$ and at the
prime divisors of $\Delta_0$ that are congruent to $3$ modulo~$4$, and the 
order $\calObar[i]$ of $\HH$ defined in Section~\ref{S:Hayashida} has 
discriminant equal to the squarefree integer $\Delta_0$. It follows 
from~\cite{AlsinaBayer2004}*{Proposition~1.54, p.~12} that $\calObar[i]$ is
an Eichler order in $\HH$ of level equal to the product of the prime divisors
of $\Delta_0$ that are congruent to $1$ modulo~$4$. 

Let $L$ and $M$ be principal polarizations on $E^2$ such that $(E^2,L)$ and
$(E^2,M)$ are isomorphic to $\Jac C$ and $\Jac D$ as principally-polarized 
surfaces, and let $I$ and $J$ be the right ideal classes of $\calObar[i]$ 
corresponding to $L$ and $M$ under the equivalence of categories described in 
Section~\ref{S:Hayashida}. We claim that there is an element $\alpha$ of $\HH$ 
such that $\alpha I \subseteq J$ and such that the index of $\alpha I$ in $J$
is a power of~$2$. To see this, we use the following result, found in the
section of~\cite{Voight} devoted to applications of the strong approximation
theorem (paraphrased here with slightly different notation):
\begin{quote}
   Let $\HH$ be a definite quaternion algebra over a totally real field $F$ 
   with ring of integers $R$, and let $O$ be an $R$-order in~$\HH$ such that
   for all finite primes $\frakp$ of $R$, the local norm maps 
   $O_\frakp^*\to R_\frakp^*$ are surjective. Suppose the narrow class group of
   $R$ is trivial, and let $\frakp$ be a prime of $R$ which is unramified 
   in $\HH$. Then every ideal class of $O$ contains an integral $O$-ideal 
   whose reduced norm is a power of $\frakp$.
\end{quote}
Our order $\calObar[i]$ is an Eichler order, so the condition on the local 
norm maps is satisfied for all primes $p$ of $\ZZ$. Also, $2$ is unramified 
in $\HH$. 

The ideal $J$ is invertible because the order $\calObar[i]$ is hereditary
(because its reduced discriminant is squarefree).  Consider the right order 
$O$ of the lattice $I J^{-1}$; it is locally isomorphic to $\calObar[i]$ at 
every prime, so it also satisfies the condition on norm maps given above.  
Applying the quoted result with $\frakp = 2$,  we find that there is an 
$\alpha\in \HH$ such that $\alpha I J^{-1}$ is an integral ideal of norm 
$2^i$ for some $i>0$. In particular, $\alpha I \subseteq J$, and  the index
of $\alpha I$ in $J$ is a power of~$2$.

Translating this back into the category $\pol E^2$, we find that there is a 
$P\in \Hom(L,M)$ of determinant $2^i$ such that 
\[P^* M P = (\bar{\det P}) L. \]
In terms of abelian surfaces, this means that we have an isogeny $\varphi$ 
of degree $2^i$ from $\Jac C$ to $\Jac D$  such that the pullback $\varphi^* M$
of the principal polarization on $\Jac D$ is equal to $2^i$ times the 
polarization~$L$.  We will show that $\varphi$ can be written as a composition
of  generalized Richelot isogenies.

We note that the kernel of $\varphi$ is a maximal isotropic subgroup of the
$2^i$-torsion of $\Jac C$. Now there are two possibilities: Either 
$(\ker\varphi)\cap(\Jac C)[2]$ is all of $(\Jac C)[2]$, or it is an order-$4$
subgroup.

Suppose $(\ker\varphi)\cap(\Jac C)[2] = (\Jac C)[2]$.  Note that 
$(\Jac C)[2](\Fqbar)$ is a vector space over $k := \calO/2\calO \cong \FF_4$,
and the Weil pairing on $(\Jac C)[2]$ is \emph{semi-balanced} with respect to
this action of $k$; that is, we have $e_2(\alpha P, Q) = e_2(P, \bar{\alpha} Q)$
for every $P,Q\in(\Jac C)[2](\Fqbar)$ and $\alpha\in k$, where $\bar\alpha$ is
the conjugate of $\alpha$ over $\FF_2$. Then~\cite{Howe1995}*{Lemma~7.3, p.~2378}
shows that there is a one-dimensional isotropic $k$-subspace $G$ of 
$(\Jac C)[2]$. The group $G$ is a maximal isotropic subgroup of $(\Jac C)[2]$,
so we obtain a generalized Richelot isogeny whose kernel is contained in the
kernel of~$\varphi$. In other words, $\varphi$ factors through a generalized 
Richelot isogeny.

On the other hand, suppose $(\ker\varphi)\cap(\Jac C)[2]$ has order~$4$.  Then
$(\ker\varphi)(\Fqbar)$ is isomorphic as an abelian group to 
$\ZZ/2^i\ZZ \times \ZZ/2^i\ZZ$. Let $P$ and $Q$ be generators of 
$(\ker\varphi)(\Fqbar)$. Since $\ker\varphi$ is an isotropic subgroup of
$(\Jac C)[2^i]$, we have $e_{2^i}(P,Q) = 1$, where $e_{2^i}$ is the Weil
pairing on $(\Jac C)[2^i]$.  From the compatibility of the Weil pairing, it 
follows that $e_2(2^{i-1}P, 2^{i-1}Q) = 1$.  Let 
$G = (\ker\varphi)\cap(\Jac C)[2]$, so that $G(\Fqbar)$ is generated by 
$2^{i-1}P$ and $2^{i-1}Q$. We see that $G$ is a maximal isotropic subgroup of 
$(\Jac C)[2]$, and arguing as in the preceding paragraph, we find that  
$\varphi$ factors through a generalized Richelot isogeny.

In either case, $\varphi$ factors through a generalized Richelot isogeny.  
Repeating this argument, we find that $\varphi$ is in fact a composition of  
generalized Richelot isogenies.

This almost, but not quite, shows that after Step~\ref{genus2-step4} the set 
$S$ from Algorithm~\ref{A:genus2} contains all curves whose Jacobians are 
isomorphic (as unpolarized surfaces) to $E^2$.  The lacuna in the argument is 
that the sequence of generalized Richelot isogenies from $(E^2,L)$ to $(E^2,M)$
may pass through singular curves, as discussed above. Suppose this is the case,
and consider the split polarized surface $E_1\times E_2$ closest to $(E^2,M)$
along the given path of generalized Richelot isogenies. The first surface 
\emph{after} $E_1\times E_2$ will be a genus-$2$ curve obtained via the 
Howe--Lepr\'evost--Poonen 
construction~\cite{HoweLeprevostEtAl2000}*{Proposition~4, p.~324}, and so will 
appear in the set $S$ after step~\ref{genus2-step2}. The path of generalized 
Richelot isogenies from $E_1 \times E_2$ to $(E^2,M)$ does \emph{not} contain 
any further split polarized surfaces, so by the end of step~\ref{genus2-step4},
the set $S$ constructed by Algorithm~\ref{A:genus2} will contain $(E^2,M)$.
\end{proof}

\begin{remark}
In practice it can be helpful to modify Step~\ref{genus2-step2}(a) of 
Algorithm~\ref{A:genus2}.  In addition seeding the list $S$ with curves whose
Jacobians are $(2,2)$-isogenous to products of two elliptic curves of 
defect~$d$, we can also throw in curves whose Jacobians are $(3,3)$-isogenous
to a product of such elliptic curves, by using Algorithm~5.4 (p.~185) 
of~\cite{BrokerHoweEtAl2015}.
\end{remark}

\begin{expectation}
\label{HE:genus2}
Fix $d\ge 0$.  As $q$ varies over the odd prime powers, 
Algorithm~\textup{\ref{A:genus2}}, applied to $q$ and the list of trace-$d$ 
elliptic curves over $\Fq$, runs in time $\Otilde(q^{3/4})$.  Furthermore, 
if $d>0$ and $q$ is $d$-unexceptional, the algorithm produces $q^{3/4}$ curves, 
up to logarithmic factors.
\end{expectation}

\begin{proof}[Justification]
For a fixed~$d$, Heuristic~\ref{H:smalldefect2} suggests that the number of 
curves produced by the algorithm is bounded above by $q^{3/4}$, up to 
logarithmic factors.  When $d>0$ and $q$ is $d$-unexceptional, we expect the
number of curves is bounded below by a similar expression.  The time taken by
the algorithm is the size of its output, times factors of $\log q$.
\end{proof}

\section{Genus-\texorpdfstring{$4$}{4} curves with small defect}
\label{S:genus4}

In this section we present our algorithm for producing genus-$4$ curves with 
small defect.

\goodbreak
\begin{algorithm}
\label{A:genus4}
\begin{algtop}
\algin   An odd prime power $q = p^e$.
\algout  A genus-$4$ curve over $\Fq$, or the word ``failure''.
\end{algtop}
\begin{alglist}
\item Compute $m = \lfloor 2\sqrt{q} \rfloor$, set $d = 0$, and set $\calL = \{\}$.
\item \label{GENUS4-step2}
      Set $t = d - m$.  If $p\mid t$, then skip to Step~\ref{GENUS4-tailend}.
\item Set $\Delta = t^2 - 4q$ and write $\Delta = \Delta_0 F^2$ for a 
      fundamental discriminant $\Delta_0$.
\item Using the algorithm of~\cite{Sutherland2011}, compute the mod-$q$ 
      reductions of the Hilbert class polynomials of discriminant
      $\Delta_0 f^2$ for all divisors $f$ of $F$.
\item Compute the roots in $\Fq$ of these Hilbert class polynomials.
\item Compute representatives of all of the isomorphism classes of elliptic
      curves whose $j$-invariants are among these roots, and let $\calE_d$ be 
      the subset of those elliptic curves whose defect is $d$.
\item \label{GENUS4-addE}
      Add the elements of $\calE_d$ to the set $\calL$.      
\item \label{GENUS4-genus2}
      Run Algorithm~\ref{A:genus2} with inputs $q$ and $\calL$.
\item \label{GENUS4-eachC}
      For each curve $C$ in the output of Algorithm~\ref{A:genus2}:
\begin{algsublist}
\item Write $C$ as $y^2 = f$ for a sextic polynomial $f$.
\item For each factorization of $f$ into a pair $f_1,f_2$ of cubics (up to
      order and up to scaling by squares in $\Fq$), run 
      Algorithm~\ref{A:construction} on $q$, $f_1$, $f_2$, $\calL_1$, and
      $\calL_2$, where each $\calL_i$ is the set of curves in $\calL$ that
      are compatible with $f_i$.
\item If Algorithm~\ref{A:construction} outputs an element $a\in\PP^1(\Fq)$,
      output the curve $D_a(f_1,f_2)$ from Section~\textup{\ref{S:family}}
      and stop.
\end{algsublist}
\item \label{GENUS4-tailend}
      Increment $d$. If $d>m$, output ``failure'' and stop.  Otherwise, go to 
      step~\ref{GENUS4-step2}.
\end{alglist}
\end{algorithm}

\begin{remark}
In Step~\ref{GENUS4-step2} we avoid the case $p\mid t$ solely to make the
analysis of the algorithm simpler. In actual practice, we will encounter the 
case $p\mid t$ most often when $q$ is a square and $d=0$.  In this case, we 
should simply compute the set $\calE_d$ of (supersingular) elliptic curves of 
defect $0$ in any of a number of ways --- by using the formulas
from~\cite{KanekoZagier1998}, for example, or by using the algorithm of
Br\"oker~\cite{Broker2009} to compute one such curve and then computing the 
graph of $2$-isogenies --- and then continue with Step~\ref{GENUS4-addE}.
\end{remark}

\begin{expectation}
\label{HE:genus4}
For large odd $1$-unexceptional prime powers~$q$, 
Algorithm~\textup{\ref{A:genus4}} will output a curve of defect at most $4$ 
in time $\Otilde(q^{3/4})$.
\end{expectation}

The justification of this heuristic expectation depends on knowing something 
about the Galois structure of the Weierstrass points of the genus-$2$ curves 
produced by Algorithm~\ref{A:genus2}; we need to know that some fraction of 
these curves can be used in Algorithm~\ref{A:construction}.  
Proposition~\ref{P:splittings} below gives us the information we need.  Let us
begin by setting up the notation for the proposition.

Let $E$ be an ordinary elliptic curve over a finite field~$\Fq$ of odd 
characteristic and let $\pi$ be the Frobenius endomorphism of~$E$. Let $R$ be 
the subring $\ZZ[\pi]$ of $\End E$.  From~\cite{Deligne1969} we know that the 
abelian surfaces isogenous to $E^2$ are in bijection with the torsion-free
$R$-modules of rank~$2$, and results of Borevi\v{c} and
Faddeev~\cite{BorevichFaddeev1965} (summarized 
in~\cite{Kani2011}*{Theorem~48, p.~326}) classify such $R$-modules. Pushing
this classification back through Deligne's result, we find the following:
Every abelian surface $A$ isogenous to $E^2$ can be written as $E_1\times E_2$ 
for two elliptic curves with $\End E_1 \supseteq \End E_2$; if $A$ can also be 
written $E_1'\times E_2'$ where $\End E_1' \supseteq \End E_2'$, then 
$\End E_1' \cong \End E_1$ and $\End E_2' \cong \End E_2$; and, given any 
elliptic curve $E_1'$ with $\End E_1' \cong \End E_1$, there is a unique
$E_2'$ such that $A \cong E_1'\times E_2'$.

Suppose $A$ is isogenous to $E^2$, and write $A \cong E_1\times E_2$ as above.
Note that the conductor of the quadratic order $\End E_1$ divides that of 
$\End E_2$; it follows that the dimension $d_1$ of the $\FF_2$-vector space
$E_1[2](\Fq)$ is greater than or equal to the dimension $d_2$ of $E_2[2](\Fq)$.
Also, since $E_1(\Fq)$ has even order if and only if $E_2(\Fq)$ has even order,
we have $d_1 = 0$ if and only if $d_2 = 0$.

\begin{proposition}
\label{P:splittings}
With notation as above, let $\calC$ be the isogeny class of $E$ and let $\calS$ 
be the set of genus-$2$ curves over $\Fq$ whose Jacobians are isomorphic 
{\upleftparen}as unpolarized surfaces{\uprightparen} to $A$. 
\begin{enumerate}
\item If $(d_1,d_2) = (0,0)$, then every curve in $\calS$ can be written in 
      the form $y^2 = f_1 f_2$ for irreducible cubic polynomials 
      $f_1,f_2\in\Fq[x]$, and every curve in $\calC$ is compatible with both 
      $f_1$ and~$f_2$.
\item If $(d_1,d_2) = (2,2)$, then every curve in $\calS$ can be written in the
      form $y^2 = f_1 f_2$ for cubic polynomials $f_1,f_2\in\Fq[x]$ that are 
      completely split, and at least $1/4$ of the curves in $\calC$ are 
      compatible with both $f_1$ and~$f_2$.
\item If $(d_1,d_2) = (2,1)$, then every curve in $\calS$ can be written in the
      form $y^2 = f_1 f_2$ for cubic polynomials $f_1,f_2\in\Fq[x]$, where 
      $f_1$ has only one root and $f_2$ is completely split\textup{;} at least
      $1/2$ of the curves in $\calC$ are compatible with~$f_1$, and at least
      $1/4$ are compatible with~$f_2$.
\item Suppose  $(d_1,d_2) = (1,1)$. If $C$ is a curve in $\calS$, then $C$ 
      can either be written in the form $y^2=f_1 f_2 f_3$ for three irreducible 
      quadratic polynomials $f_1,f_2,f_3\in\Fq[x]$, or in the form 
      $y^2 = f_1 f_2$, where $f_1$ and $f_2$ are cubic polynomials, each with 
      exactly one root. At least $1/4$ of the curves in $\calS$ are of the 
      latter type\textup{;} and, for these curves, at least $1/2$ of the curves 
      in $\calC$ are compatible with both $f_1$ and~$f_2$.
\end{enumerate}
\end{proposition}

\begin{proof}
We note for future reference that if the curves in $\calC$ have even group 
orders, then the theory of isogeny volcanoes shows that either no curves in 
$\calC$ have $2$-rank equal to $2$, or at least $1/4$ of them do.  Similarly, 
at least $1/2$ of the curves in $\calC$ have $2$-rank equal to~$1$.

Suppose $C$ is a curve in $\calS$.  The six Weierstrass points of $C$ fall into
orbits under the action of the absolute Galois group of $\Fq$, and the orbit 
structure determines the ranks of the $2$-torsion subgroup of $\Jac C$ over the
extensions of~$\Fq$.  These ranks are also determined by the pair $(d_1,d_2)$. 
Comparing these ranks (for the first three extensions of $\Fq$) for the various 
possible Galois structures and the various possible pairs $(d_1,d_2)$, we find
the following:

The only Galois orbit structure compatible with $(d_1,d_2) = (0,0)$ is for the
Weierstrass points to be divided into two orbits of size $3$;  this translates 
into $C$ being of the form $y^2 = f_1 f_2$ for two irreducible cubics $f_1$ and 
$f_2$.  The curves in $\calC$ have no rational points of order~$2$, so $f_1$ 
and $f_2$ are compatible with all of the curves in~$\calC$.

Similarly, the only Galois orbit structure that is compatible with 
$(d_1,d_2) = (2,2)$ is six orbits of size~$1$; this means that $C$ can be
written (in several ways) as $y^2 = f_1 f_2$, for two completely split cubics
$f_1$ and $f_2$. Since $E_1$ and $E_2$ both have $2$-rank equal to $2$, we see
that at least $1/4$ of the curves in $\calC$ have $2$-rank equal to~$2$, and 
each such curve is compatible with both $f_1$ and $f_2$.

The only Galois orbit structure compatible with $(d_1,d_2) = (2,1)$ is one 
orbit of size $2$ and four of size $1$; this means that $C$ can be written
(in several ways) as $y^2 = f_1 f_2$, where $f_1$ is a cubic with only one root
and $f_2$ is a completely split cubic.  At least $1/2$ of the curves in $\calC$
have $2$-rank equal to $1$ and are compatible with~$f_1$; and, since $E_1$ has 
$2$-rank equal to~$2$, at least $1/4$ of the curves in $\calC$ have $2$-rank 
equal to $2$ and are compatible with $f_2$.

We are left to consider the case $(d_1,d_2) = (1,1)$. There are two Galois 
orbit structures compatible with these values of $d_1$ and~$d_2$: three orbits 
of size $2$, or two orbits of size $2$ and two of size~$1$. To analyze this 
case, we consider a graph $G$ constructed as follows.

We let the vertices of $G$ be the isomorphism classes of principal 
polarizations on~$A$.  Given two principal polarizations $\lambda$ and $\mu$,
we connect the associated vertices with an edge if and only if there is a diagram
\begin{equation}
\label{EQ:Richelot}
\xymatrix{
A \ar[r]^{2\lambda}\ar[d]_\Phi & \hat{A} \\
A \ar[r]^\mu                   & \hat{A}\ar[u]_{\hat{\Phi}}.\\
}
\end{equation}
Each edge from a vertex $\lambda$ gives rise to a Galois-stable subgroup of 
$A[2](\Fqbar)$ that is maximal isotropic with respect to the Weil pairing 
associated to~$\lambda$ --- namely, $\ker \Phi$.

Let us call a polarization $\lambda$ of $A$ \emph{bad} if the polarized variety
$(A,\lambda)$ is isomorphic to the Jacobian of a curve $C$ whose Weierstrass 
points form three Galois orbits of size $2$; \emph{good} if the polarized 
variety $(A,\lambda)$ is isomorphic to the Jacobian of a curve $C$ whose 
Weierstrass points form two Galois orbits of size $2$ and two of size~$1$; and 
\emph{split} if the polarized variety $(A,\lambda)$ is isomorphic to the
product of two elliptic curves with the corresponding product polarization.
We will show that in the graph~$G$, every bad vertex is adjacent to at least
one good vertex, and every good vertex is connected to at most three bad 
vertices.  From this it follows that there are at most three times as many bad
polarizations as good ones, and therefore at least $1/4$ of the curves in 
$\calS$ can be written $y^2 = f_1 f_2$ for cubics $f_1$ and $f_2$, each with 
exactly one root.

Any morphism from $A$ to its dual variety that is equal to its own dual 
morphism can be represented by a $2\times2$ array 
\[ L = \begin{bmatrix} k & \bar{\alpha} \\ \alpha & \ell\end{bmatrix} ,\]
where $k$ and $\ell$ are integers, $\alpha$ is a homomorphism from $E_1$ 
to~$E_2$, and $\bar{\alpha}$ is the dual morphism from $E_2$ to~$E_1$. Such an
array gives an endomorphism of $E_1\times E_2$; composed with the product 
polarization from $E_1\times E_2$ to its dual, this endomorphism gives a
polarization if $k$, $\ell$, and $k\ell - \degree \alpha$  are all positive,
and a principal polarization if $k\ell - \degree\alpha = 1$.

We will show that a polarization is bad if and only if the array associated to
it has $k$ and $\ell$ both even.  To do this, we will count the number of
Galois-stable subgroups of $A[2](\Fqbar)$ that are maximal isotropic subgroups 
with respect to the Weil pairing induced by the polarization.  First we count
the number of such subgroups for good, bad, and split curves, and then we count
these subgroups for polarizations of $A$ given by arrays as above.

Suppose $\lambda$ is a bad polarization, corresponding to a curve~$C$.  The
Galois-stable maximal isotropic subgroups of $(\Jac C)[2](\Fqbar)$ correspond 
to Galois-stable partitions of the six Weierstrass points of $C$ into three 
disjoint subsets of size two.  For a bad curve, it is easy to check that there 
are seven such partitions.

Similarly, if $\lambda$ is a good polarization, we find that there are exactly 
three Galois-stable maximal isotropic subgroups of the $2$-torsion.  If 
$\lambda$ is a split polarization, then again there are exactly three such 
subgroups.  This shows that good curves and split curves are connected to at 
most three other vertices in the graph $G$, and in particular are connected to 
at most three bad vertices.

Now let us count the Galois-stable subgroups of $A[2](\Fqbar)$ that are maximal
isotropic with respect to the Weil pairing obtained from a polarization 
described by an array as above.  First we simply count the Galois-stable 
subgroups of order four, \emph{without} the isotropy condition.

For each $i$ let $P_i$ be the rational $2$-torsion point of $E_i$ and let 
$Q_i$ be a non-rational $2$-torsion point of $E_i$. Note that $A[2](\Fqbar)$
contains exactly seven Galois-stable subgroups of order $4$, namely:
\begin{multline}
\label{EQ:kernels}
\langle (P_1,0), (Q_1,0)\rangle,\quad
\langle (0,P_2), (0,Q_1)\rangle,\\
\langle (P_1,0), (Q_1,P_2)\rangle,\quad
\langle (0,P_2), (P_1,Q_2)\rangle,\quad
\langle (P_1,P_2), (Q_1,Q_2)\rangle,\\
\langle (P_1,P_2), (Q_1,Q_2+P_2)\rangle,\quad
\langle (P_1,0), (0, P_2)\rangle.
\end{multline}
We check that all seven of these subgroups are isotropic with respect to the 
Weil pairing on the $2$-torsion associated to a polarization $\lambda$ if and
only if the array $L$ associated to $\lambda$ has $k$ and $\ell$ both even.
Thus, the bad polarizations are precisely the polarization that can be
represented by endomorphisms 
\[ L = \begin{bmatrix} k & \bar{\alpha} \\ \alpha & \ell\end{bmatrix} \]
such that both $k$ and $\ell$ are even.

Suppose $\lambda$ is bad, represented by an $L$ as above with $k$ and $\ell$ 
even.  First, we will show that there is a polarization isomorphic to $\lambda$
whose associated $L$ has $k\equiv 2 \bmod 4$ or $\ell \equiv 2\bmod 4$.

Certainly if $k\equiv 2 \bmod 4$ or $\ell \equiv 2\bmod 4$ we are done, so
assume that both $k$ and $\ell$ are divisible by $4$.  Since 
$k\ell - \degree\alpha = 1$, we see that $\degree\alpha$ is odd. Consider the 
automorphism 
\[
P = \begin{bmatrix} 1 & 0 \\ \alpha & 1 \end{bmatrix} \]
of $A$.  Pulling back $\lambda$ via $P$ replaces $L$ with 
\begin{align*}
(P^*)^{-1} L P^{-1} & = 
\begin{bmatrix} 1 & -\bar{\alpha} \\ 0       & 1   \end{bmatrix}
\begin{bmatrix} k &  \bar{\alpha} \\ \alpha  & \ell\end{bmatrix}
\begin{bmatrix} 1 & 0             \\ -\alpha & 1   \end{bmatrix}\\
& = 
\begin{bmatrix} k + (\ell - 2) \bar{\alpha}\alpha  & (1 - \ell) \bar{\alpha} \\ 
                (1 - \ell) \alpha                  & \ell  \end{bmatrix}.
\end{align*}                
Since $\bar{\alpha}\alpha = \degree\alpha$ is odd and $k$ and $\ell$ are 
divisible by~$4$, the upper left entry in this new array is congruent to $2$ 
modulo~$4$. Thus, if a polarization is bad, it is isomorphic to a bad 
polarization $\lambda$ for which either $k$ or $\ell$ is congruent to $2$
modulo~$4$.

Suppose $\lambda$ is a bad polarization with $k \equiv 2 \bmod 4$.  Consider 
the polarization $\mu$ we obtain by taking $\Phi$ in diagram~\eqref{EQ:Richelot}
to have kernel equal to the first group in \eqref{EQ:kernels}; that is,
we take $\Phi$ to be
$\left[\begin{smallmatrix} 2 & 0 \\ 0 & 1   \end{smallmatrix}\right]$.  
We compute that $\mu$ is given by the endomorphism 
\[M = \begin{bmatrix} k/2 & \bar{\alpha} \\ \alpha & 2\ell\end{bmatrix},\]
whose upper left entry is odd, so $\mu$ is either good or split.  Likewise, if 
$\ell \equiv 2 \bmod 4$, we can take  $\Phi$ in diagram~\eqref{EQ:Richelot}
to have kernel equal to the second group in \eqref{EQ:kernels}; that is, we 
take $\Phi$ to be 
$\left[\begin{smallmatrix} 1 & 0 \\ 0 & 2   \end{smallmatrix}\right]$.  Then
$\mu$ is given by the endomorphism 
\[M = \begin{bmatrix} 2k & \bar{\alpha} \\ \alpha & \ell/2\end{bmatrix},\]
whose lower right entry is odd, so $\mu$ is either good or split. This shows
that in the graph $G$, every bad vertex is connected to at least one vertex
that is good or split.  To complete our argument, we need only show that bad 
vertices cannot be adjacent to split vertices.

If $\mu$ is a split polarization, then the polarized variety $(A,\mu)$ is 
isomorphic to a product surface $F_1\times F_2$ with the product polarization,
and both $F_1$ and $F_2$ have $2$-rank equal to $1$.  If $\mu$ is connected to
a good or a bad polarization, corresponding to a curve $C$, then $C$ is 
obtained from $F_1$ and $F_2$ using the Howe--Lepr\'evost--Poonen  
construction~\cite{HoweLeprevostEtAl2000}*{Proposition~4}. Since $F_1$ and 
$F_2$ have rank $1$, they have models of the form
\[
F_1\colon \quad y^2 = (x^2 - n)(a_1 x - b_1) \qquad
F_2\colon \quad y^2 = (x^2 - n)(a_2 x - b_2),\]
where $n$ is a nonsquare in $\Fq$.  One computes that the curves obtained 
from $F_1$ and $F_2$ using the formulas 
of~\cite{HoweLeprevostEtAl2000}*{Proposition~4} are of the form $y^2 = h$, 
where $h$ is a sextic polynomial having $1$ and $-1$ as roots.  Thus, the 
vertices adjacent to a split vertex are either split or good.  In particular, 
a bad vertex is never adjacent to a split vertex.  
\end{proof}

\begin{proof}[Justification of Heuristic Expectation~\textup{\ref{HE:genus4}}]
Take $d = 1$.  We already noted that, under the GRH, the number of defect-$1$
elliptic curves over $\Fq$ for $1$-unexceptional $q$ grows like $q^{1/4}$, up
to logarithmic factors, and Heuristic Expectation~\ref{HE:genus2} tells us to 
expect the number of genus-$2$ curves over $\Fq$ of defect $2$ produced by 
Algorithm~\ref{A:genus2} to grow like $q^{3/4}$, up to logarithmic factors.
Proposition~\ref{P:splittings} tells us that at least $1/4$ of the curves
produced by the algorithm can be written as $y^2 = f_1 f_2$ for cubic 
polynomials $f_1$ and $f_2$ that are each compatible with at least $1/4$ of the
defect-$1$ elliptic curves over~$\Fq$.  

For each genus-$2$ curve $C$ in Step~\ref{GENUS4-eachC}, we expect 
Algorithm~\ref{A:construction} to succeed with probability on the order of 
$q^{-1/2}$, so we expect to have to apply Step~\ref{GENUS4-genus2} to about 
$q^{1/2}$ curves $C$ before we succeed. Each application of
Algorithm~\ref{A:construction} takes time $\Otilde(q^{1/4})$, so the total time
to success should be $\Otilde(q^{3/4}).$
\end{proof}

\section{Results}
\label{S:results}

We implemented our algorithms in Magma, and we ran Algorithm~\ref{A:genus4}
on all of the odd prime powers less than 100,000.  (This took a few days,
running in the background on a modest laptop computer.) There are $9684$ such
prime powers~$q$, four of which --- $3^3$, $3^5$, $3^9$, and $5^5$ --- are
$1$-exceptional in the sense defined in Section~\ref{S:genus2}.  The genus-$4$
curves produced by the algorithm had
\begin{itemize}
\item defect $0$ for 3027 of these $q$ ($\approx 31.3\%$),
\item defect $2$ for 2268 of these $q$ ($\approx 23.4\%$),
\item defect $4$ for 4054 of these $q$ ($\approx 41.9\%$),
\item defect $6$ for  330 of these $q$ ($\approx  3.4\%$), and
\item defect $8$ for    5 of these $q$ ($\approx 0.05\%$).
\end{itemize}

The five $q$ for which the best curve we found had defect $8$ are
the primes
\[154^2 + 3, \quad 160^2 + 160 + 3, \quad 221^2 + 16, \quad 282^2 + 282 + 5,
\quad\text{and}\quad 307^2 + 4.\]

We maintain our conviction that for large enough $1$-unexceptional $q$, our 
algorithm will find a curve of defect $4$ or less --- but $q$ may have to be
large indeed, because even though we expect the number of genus-$2$ defect-$2$
curves to grow like~$q^{3/4}$, the implied constant is fairly small.


\end{document}